\documentclass[12pt,reqno]{article}

\usepackage[usenames]{color}
\usepackage{amssymb}
\usepackage{graphicx}
\usepackage{amscd}

\usepackage[colorlinks=true,
linkcolor=webgreen,
filecolor=webbrown,
citecolor=webgreen]{hyperref}

\definecolor{webgreen}{rgb}{0,.5,0}
\definecolor{webbrown}{rgb}{.6,0,0}

\usepackage{color}
\usepackage{fullpage}
\usepackage{float}

\usepackage{graphics,amsmath,amssymb}
\usepackage{amsthm}
\usepackage{amsfonts}
\usepackage{latexsym}
\usepackage{epsf}

\newcommand{\Spec}{\operatorname{Spec}}

\renewcommand{\epsilon}{\varepsilon}
\renewcommand{\phi}{\varphi}

\theoremstyle{plain}
\newtheorem{Theorem}{Theorem}

\newtheorem{Lemma}[Theorem]{Lemma}

\theoremstyle{definition}
\newtheorem{Definition}[Theorem]{Definition}
\newtheorem{Question}[Theorem]{Question}
\newtheorem{Example}[Theorem]{Example}

\theoremstyle{Remark}

\begin{document}

\begin{center}
\vskip 1cm{\LARGE\bf Geometrically Nilpotent Subvarieties}
\vskip 1cm
\large
Alexander Borisov  \\
Department of Mathematics \\
Binghamton University\\ 
4400 Vestal Parkway Eastl\\
Binghamton, New York 13902-6000\\
U.S.A.\\
\href{mailto:borisov@pitt.edu}{\tt borisov@math.binghamton.edu} \\
\ \\

\end{center}

\vskip .2 in

\begin{abstract} We construct some examples of polynomial maps over finite fields that admit subvarieties with a peculiar property: every geometric point is mapped to a fixed point by some iteration of the map, while the whole subvariety is not. Several related open questions are stated and discussed. 
\end{abstract}

\section{Introduction}

Suppose \(X\) is a variety defined over a finite field \(F\), and \(T:X\to X\) is a regular map, defined over \(F.\) Then for every finite extension \(K\) of \(F\) the map \(T\) induces a map \(T_K\) on the set of geometric points of \(X\) over \(K\). These maps commute with the action of the Galois group of \(K\) over \(F\) and the iterations of these maps are induced by the iterations of \(T\).

Of special interest is the case when \(X\) is the affine space. Here \(T\) is a polynomial map. It is given by \(n\) polynomials in \(n\) variables defined over \(F\). One particular source of these maps is reduction modulo \(p\) of integer polynomial maps. Such maps have been studied extensively in special cases, often in connection with factorization algorithms and cryptology. Generally speaking, this is an area of many questions and few complete answers \cite{BGHKST, ThreeMaps, VasigaShallit}.

In particular, it is interesting to ask when \(T_K\) is nilpotent (i.e. some power of it sends the entire set \(K^n\) to a single point. The following theorem is an almost immediate corollary of the density theorem of Borisov and Sapir \cite{BS}.

\begin{Theorem} Suppose \(T\) is a polynomial map defined over a finite field \(F.\) The \(T_K\) is nilpotent for all \(K\) if and only if \(T\) is nilpotent.
\end{Theorem}

\begin{proof} The ``if'' part is obvious. For the converse, suppose \(V\) is the Zariski closure of \(T^{(n)}(A^n)\), where \(T^{(n)}\) is the \(n\)-th iteration of \(T\). Then \(V\) is invariant under \(T\), the restriction of \(T\) to \(V\) is dominant, and periodic geometric points of \(T\) (more precisely, quasi-fixed points) are dense in \(V\) \cite{BS}. Suppose \(\dim V \geq 1\). Take any two distinct periodic points on \(V\). Then for any \(K\) that contains both of their fields of definition, the map \(T_K\) is not nilpotent.
\end{proof}

In contrast to the above mentioned density theorem, Borisov constructed a simple two-variable polynomial map \(T\) over integers, such that all of its reductions modulo rimes \(p\) are dominant, but the maps \(T_{\mathbb F_{p}}\) are  nilpotent  \cite{ThreeMaps}. Of course, this does not contradict the density theorem: the maps \(T_K\) have many periodic points for extensions \(K\) of \({\mathbb F_{p}}\).

The following two definitions are crucial for this paper.

\begin{Definition}  Suppose \(T:X\to X\) is a regular map over a finite field \(F\), and \(Y\) is a subvariety of \(X\). Then \(Y\) is a nilpotent subvariety of \(X\) with  respect to \(T\) if for some point \(P\) of \(X\), fixed by \(T\), and some integer \(k \geq 1,\) \(T^{(k)}(Y)=\{P\}.\)
\end{Definition}

\begin{Definition} Suppose \(T:X\to X\) is a regular map over a finite field \(F\), and \(Y\) is a subvariety of \(X\). Then \(Y\) is a geometrically nilpotent subvariety of \(X\) with  respect to \(T\) if for some point \(P\) of \(X\), fixed by \(T\), all geometric points of \(Y\) are mapped to \(P\) by a sufficiently high iteration of \(T.\)
\end{Definition}

 It is important to note that in both of the above definitions  \(Y\) is not required to be \(T\)-invariant.  Clearly, if \(Y\) is nilpotent, then it  is geometrically nilpotent. The main goal of this paper is to show that the converse statement is false. In fact, we give two different constructions, that provide infinitely many counterexamples.  The first one combines  the trap construction of Borisov \cite{ThreeMaps} with the Pollard's \(\rho\) algorithm \cite{Pollard}. The second one replaces the Pollard 's \(\rho\) algorithm by the periodicity of the Fibonacci sequence in finite abelian groups.   The existence of these examples opens several interesting questions that we discuss in the last section.

\section{Examples and Theorems}

\begin{Example} \label{e1} Suppose \(F\) is a finite field, and \(a\in F\). Define the map \(T\) from the \(X=A^{3}(F)\) to itself as follows:
\[T(x,y,z)=\left((x^2+az^2)(x-y)z^3,\ ((y^2+az^2)^2+az^4)(x-y)z,\ (x-y)z^5\right).\]
Define \(Y\subset X\) by the equation \(x^2+az^2=yz\).
\end{Example}

Note that in the above example \(T(0,0,0)=(0,0,0)\).
\begin{Theorem} For every geometric point \(P\) of \(Y,\) defined over a finite extension \(K\) of \(F,\) there exists \(k\geq 1\) such that \(T_K^{(k)}(P)=(0,0,0).\) However no single \(k\) works for all geometric points of \(Y.\)
\end{Theorem}

\begin{proof} Suppose \(P=(x_0,y_0,z_0),\) and \(T_K^{(k)}(P)=(x_k,y_k,z_k)\) for all \(k\geq 1\). If \(z_0=0,\) then \(x_1=y_1=z_1=0\). If \(z_0\neq 0,\) note that the polynomials defining \(T\) are homogeneous of the same degree. Whenever \(z\neq 0\), define \(u=\frac{x}{z}\) and \(v=\frac{y}{z}\), and, in particular \(u_k=\frac{x_k}{z_k}, v_k=\frac{y_k}{z_k}\). Note that, whenever defined, \((u_{k+1},v_{k+1})=g(u_k,v_k),\) where
\[g(u,v)=(u^2+a, (v^2+a)^2+a)\]
This polynomial map can be expressed as \((u,v)\mapsto (h(u),h(h(v)))\) where \(h(t)=t^2+a\). Note also that for any point in \(Y\) we have \(v_0=h(u_0)\). Because \(K\) is finite, iterations of \(h\) on the points of \(K\) starting from \(u_0,\) must end up in a cycle (possibly of length one). Since \(v_k\) ``goes twice as fast" as \(u_k\), for some \(k\) we will have \(u_k=v_k\) \cite{Pollard}. For this value of \(k\) we will have \(x_k=y_k.\) Therefore \(z_{k+1}=0,\) thus \(x_{k+2}=y_{k+2}=z_{k+2}=0.\)

Note that different iteration powers of of \(h(t)\) have different degrees and thus are different. So no single \(k\) works for all geometric points of \(Y.\)
\end{proof}

The above theorem implies that \(Y\) is geometrically nilpotent but not nilpotent. Note that the construction is quite universal. Instead of \(h(t)=t^2+a,\)  one can take almost any polynomial. The choice of  \(h(t)=t^2+a\) was made partially to emphasize that almost any \(h(t)\) works and partially because its dynamics for \(a=0\) and \(a=-2\) has been extensively studied before \cite{VasigaShallit}. Note also that if we take \(a\in \mathbb Z,\) the same formulas can be interpreted as an integer polynomial map, and a subscheme \(Y\), that has geometrically nilpotent non-nilpotent reductions modulo every prime \(p\).

If instead of the polynomial \(f(t)=t^2+a\) we take \(f(t)=t+1,\) we get another interesting map.

\begin{Example} \label{e2} Define an integer polynomial map \(T\) by the formula 
\[T(x,y,z)=\left( (x+z)(x-y)z,\ (y+2z)(x-y)z,\ (x-y)z^3  \right).\]
Define a subscheme \(Y\) by an equation \(x+z=y.\)
\end{Example} 

\begin{Theorem} For every prime \(p\) the above map \(T\) the reduction of \(T^{(p+1)}\) modulo \(p\) contracts the reduction modulo \(p\) of \(Y\) to the point \((0,0,0).\)
\end{Theorem}

\begin{proof}  Fix a priime \(p\). Slightly abusing the notation for the sake of brevity, we will denote by \(T\) and \(Y\) their reductions modulo \(p\). As in the previous theorem, we define \(u_k=\frac{x_k}{z_k}\) and \(v_k=\frac{y_k}{z_k}\). Then for all geometric  points \((x_0,y_0,z_0)\)  in \(Y\) we have 
\[v_0=u_0+1,\ v_1=u_1+2\ ,...\ , v_{p-1}=u_{p-1}+p.\]
So \(v_{p-1}=u_{p-1},\) thus  \(y_{p-1}=x_{p-1}.\) Therefore, \(z_p=0\) and \(x_{p+1}=y_{p+1}=z_{p+1}=0.\)
\end{proof}

From the proof, it is clear that no smaller power of the reduction of \(T\) contracts the entire \(Y\) to a point. In fact, \(T^{(k)}(Y)\) is a two-dimensional subscheme of \((\Spec \mathbb  Z)^3\) for every \(k\geq 1\).

The above construction is by no means the only trick that one can play to construct geometrically nilpotent non-nilpotent subvarieties. The following example uses the periodic nature of the Fibonacci sequence in finite abelian groups instead of the Pollard's \(\rho\) algorithm.

\begin{Example}  Define an integer polynomial map \(T\) by the formula 
\[T(x,y,z)=\left( y(x-1)z^2,\ xy(x-1)z,\ (x-1)z^3  \right).\]
Define a subscheme \(Y\) by the equation \(x=y.\)
\end{Example}

\begin{Theorem} For every prime \(p\) the reduction of \(Y\) modulo \(p\) is geometrically nilpotent but not nilpotent for the reduction of \(T\) modulo \(p\).
\end{Theorem}

\begin{proof} Fix a priime \(p\). Slightly abusing the notation for the sake of brevity, we will denote by \(T\) and \(Y\) their reductions modulo \(p\). As before, we denote  \(u_k=\frac{x_k}{z_k}\) and \(v_k=\frac{y_k}{z_k}\). Then for all geometric  points \((x_0,y_0,z_0)\)  in \(Y\) we have, as long as \(z_k\neq 0,\)  \((u_{k+1},v_{k+1})=(v_k,u_kv_k)\). We also know that \(u_0=v_0,\) for all points in \(Y.\) This means that \((u_k,v_k)=(a_k,a_{k+1}),\) where \(\{a_k\}\) satisfies Fibonacci-like recursion in the multiplicative group of the field \(K.\) The following lemma is well known, and is included here for the convenience of the reader, and lack of a canonical reference. Note that in the lemma we write operation as addition, but we will apply it in the multiplicative notation.

\begin{Lemma} Suppose \(A\) is finite abelian group, and \(\{a_k\}_{k=0}^{\infty}\) is a sequence of elements of \(A\) so that
\[a_0=a_1; \ \ a_{k+2}=a_{k+1}+a_k\]
Then  \(a_m=0\) for some natural \(m\).
\end{Lemma}
\begin{proof} (of Lemma) The number of pairs \((a,b)\) from \(A^2\) is finite, and the map \(H:A\to A\) defined by \(H(a,b)=(b,a+b)\) is a bijection. As a permutation of \(A^2,\) it is a product of commuting cycles. Note that \((a_0,a_0)=H(0,a_0)\), so \(H^{(m)}(a_0,a_0)=(0,a_0),\) where \(m+1\) is the length of the cycle containing \((0,a_0).\)
\end{proof}

The multiplicative group \(K^*\) of every finite  field \(K\) is cyclic of order \(|K|-1.\) By the above lemma, for every points \((x_0,y_0,z_0)\) in \(Y(K)\) we get  \(u_k=0\) for some \(k\). Note that no \(k\) would work for all points for all \(K.\) Indeed, if the order of  \(K^*\) is greater than the \(k\)-th Fibonacci number (the classical one), and \(u_0\) is a generator of \(K^*\), then \(u_i\neq 1 \) for all \(i\leq k\). As in Example \ref{e1}, this implies that \(Y\) is geometrically nilpotent but not nilpotent.
\end{proof}

Note that the above example is also an integer polynomial map. Also, instead of the Fibonacci sequence, one can use virtually any rational automorphism.  One just has to be careful to avoid automorphisms of finite order, otherwise the variety \(Y\) will be not just geometrically nilpotent but actually nilpotent. For example, using the Fibonacci sequence additively gives an example similar to Example \ref{e2}.

\section{Open Questions}

Note that all examples in the previous section involve maps in three variables. It is very unlikely that there exists an integer polynomial map in two variables such that all of its reductions modulo primes have a geometrically nilpotent non-nilpotent subvariety. However, the following question is very intriguing.

\begin{Question} Does there exist a two-variable polynomial map over some finite field that has  a geometrically nilpotent non-nilpotent subvariety? On the one hand, it seems feasible that one can ``trade" one dimension for sticking to a single prime. But, on the other hand, the trap construction ``takes" one dimension, and the remaining dimension appears to be insufficient for our purposes. Perhaps this question can be answered negatively by some kind of point count, or one can introduce  a ``portable" version of the trap, that would not require a designated variable.
\end{Question}

\begin{Question} 
In all  of the above examples, the  point \(P\) was the same for all geometric points of \(Y\). It is not hard to construct examples with reducible \(Y\) whose different irreducible components have geometric points mapped to different points \(P\). But is  it possible to have several different fixed points \(P\) that would serve as images of high iterations of \(T\) of geometric points of a geometrically irreducible \(Y\)?  
\end{Question}

\begin{Question} Note that if \(Y\) is a geometrically nilpotent subvariety for \(T\) then so are all subvarieties of \((T^{(k)})^{-1}(T^{(m)}(Y))\) for all \(k,m \in \mathbb N\). Is it true that for any map \(T\) there exists a geometrically nilpotent subvariety \(Y\), possibly empty or reducible, so that every geometrically nilpotent subvariety for \(T\) is contained in \((T^{(k)})^{-1}(T^{(m)}(Y))\) for some \(k\) and \(m\)?
\end{Question}

%

\bigskip
\hrule
\bigskip

\noindent 2010 {\it Mathematics Subject Classification}:
Primary 37P25; Secondary 37P05, 14G15, 11G25.

\noindent \emph{Keywords: } 
polynomial map, reduction, iteration

\bigskip
\hrule
\bigskip

\end{document}